\documentclass[11pt,twoside, leqno]{article}
\usepackage{amsthm}
\usepackage{amssymb}
\usepackage{amsmath}
\usepackage{mathrsfs}
\usepackage{txfonts}
\usepackage{graphics}
\usepackage{epsfig}
\usepackage{multirow}
\usepackage{verbatim}

\allowdisplaybreaks

\newtheorem{thm}{Theorem}
\newtheorem{lem}[thm]{Lemma}
\newtheorem{cor}[thm]{Corollary}

\newtheorem{rem}[thm]{Remark}

\def\rr{{\mathbb R}}
\def\rn{{{\rr}^n}}

\def\lf{\left}
\def\r{\right}

\begin{document}

\arraycolsep=1pt

\author{Renjin Jiang, Jie Xiao and Dachun Yang} \arraycolsep=1pt
\title{\Large\bf Towards Spaces of Harmonic Functions with
Traces \\ in Square Campanato Spaces and Their Scaling Invariants
\footnotetext{\hspace{-0.35cm} 2010
{\it Mathematics Subject Classification}. Primary: 42B35; Secondary: 31B05, 35Q30, 42B37, 46E35.
\endgraf{\it Key words and phrases}. harmonic function, trace, square Campanato space,
scaling invariant, heat equation, Navier-Stokes system.
\endgraf J. Xiao was supported by NSERC of Canada (FOAPAL \# 202979463102000)
and URP of Memorial University (FOAPAL \# 208227463102000);
R. Jiang and D. Yang were partially supported by the National
Natural Science Foundation of China (Nos. 11301029, 11171027 \& 11361020),
the Specialized Research Fund for the Doctoral Program of Higher Education
of China (No. 20120003110003) and the Fundamental Research Funds for Central
Universities of China (Nos. 2012LYB26 \& 2013YB60).}}
\date{ }
\maketitle

\vspace{-0.8cm}

\begin{center}
\begin{minipage}{12cm}\small
{\noindent{\bf Abstract.}\, For $n\ge 1$ and $\alpha\in (-1,1)$, let $H^{\alpha,2}$ be
the space of harmonic functions $u$
on the upper half space $\mathbb{R}^{n+1}_+$ satisfying
$$\displaystyle\sup_{(x_0,r)\in \mathbb R^{n+1}_+}r^{-(2\alpha+n)}\int_{B(x_0,r)}\int_0^r|\nabla_{x,t} u(x,t)|^2t\,dt\,dx<\infty,$$
and $\mathcal{L}_{2,n+2\alpha}$ be the Campanato space on $\mathbb R^n$.
We show that $H^{\alpha,2}$ coincide with $e^{-t\sqrt{-\Delta}}\mathcal{L}_{2,n+2\alpha}$
for all $\alpha\in (-1,1)$, where the case $\alpha\in [0,1)$ was originally discovered
by Fabes, Johnson and Neri [Indiana Univ. Math. J. 25 (1976), 159-170]
and yet the case $\alpha\in (-1,0)$ was left open.
Moreover, for the scaling invariant version of $H^{\alpha,2}$, $\mathcal{H}^{\alpha,2}$, which
comprises all harmonic functions $u$ on $\mathbb R^{n+1}_+$ satisfying
$$\sup_{(x_0,r)\in\mathbb R^{n+1}_+}r^{-(2\alpha+n)}\int_{B(x_0,r)}
\int_0^r|\nabla_{x,t} u(x,t)|^2\,t^{1+2\alpha} \,dt\,dx<\infty,$$
we show that $\mathcal{H}^{\alpha,2}=e^{-t\sqrt{-\Delta}}(-\Delta)^\frac{\alpha}{2}\mathcal{L}_{2,n+2\alpha}$,
where $(-\Delta)^{\frac{\alpha}{2}}\mathcal{L}_{2,n+2\alpha}$ is the collection of all functions $f$ such that $(-\Delta)^{-\frac{\alpha}{2}}f$ are in $\mathcal{L}_{2,n+2\alpha}$.
Analogues for solutions to the heat equation are also established. As an application,
we show that the spaces $\big((-\Delta)^{\frac{\alpha}{2}}\mathcal{L}_{2,n+2\alpha}\big)^{-1}$ unify
$Q_{\alpha}^{-1}$, ${\mathrm{BMO}}^{-1}$ and $\dot{B}^{-1,\infty}_\infty$ naturally. }
\end{minipage}
\end{center}

\vspace{0.5cm}


\section{Introduction}\label{s1}
\setcounter{equation}{0}
\hskip\parindent The study of harmonic (caloric) functions on the upper half space $\rr^{n+1}_+$
and spaces of their traces on $\rr^n$ has a long history and plays important roles in harmonic analysis
and PDEs; see, for instance, \cite{FS,FJN2,FJN,FN,DX1,NY,DYZ,HMM} and references therein. In their seminal work \cite{FS},
Fefferman and Stein discovered that a function $f$ of bounded mean oscillation (BMO) is the trace of the harmonic function
$u(x,t)$ on $\rr^{n+1}_+$, $x\in \rr^n$, $t\in (0,\infty)$ and $u(x,0)=f(x)$, whenever $u$ satisfies
\begin{equation}\label{1.1}
\sup_{(x_0,r)\in \mathbb R^{n+1}_+}
r^{-n}\int_{B(x_0,r)}\int_0^r|\nabla_{x,t} u(x,t)|^2t\,dt\,dx<\infty.
\end{equation}
Above, $B(x_0,r)$ denotes the open ball $\{x\in\mathbb R^n:\ |x-x_0|<r\}$ and
$\nabla_{x,t}:=(\nabla_x,\partial_t):=(\partial_1,...,\partial_{n},\partial_t)$
denotes the total gradient. Fabes, Johnson and Neri \cite{FJN} further showed that the condition
\eqref{1.1} indeed characterizes all the harmonic functions whose traces are in BMO. Fabes and Neri \cite{FN} then
generalized this characterization to caloric functions, i.\,e., solutions to the heat equation.
Recently, Duong et al. \cite{DYZ} and Hofmann et al. \cite{HMM} have extended the study of this topic to
Schr\"odinger operators and elliptic operators, respectively.

Notice that Fabes et al. \cite{FJN} actually dealt more than just the space BMO. Precisely, for $n\ge 1$ and $\alpha\in (-1,1)$, let $H^{\alpha,2}\equiv H^{\alpha,2}(\rr^{n+1}_+)$ be the space of harmonic functions $u$
on the upper half space $\mathbb{R}^{n+1}_+:=\mathbb R^n\times (0,\infty)$ satisfying
$$\|u\|_{H^{\alpha,2}}:=\bigg(\displaystyle\sup_{(x_0,r)\in \mathbb R^{n+1}_+}
r^{-(2\alpha+n)}\int_{B(x_0,r)}\int_0^r|\nabla_{x,t} u(x,t)|^2t\,dt\,dx\bigg)^\frac12<\infty.$$
In what follows, let $\nabla:=\nabla_x$, $\Delta:= \sum_{j=1}^n\partial_j^2$
and denote the Poisson semigroup $P_t$ by $e^{-t\sqrt{-\Delta}}$. When $\alpha=0$,
$H^{0,2}$ is the space ${\mathrm{HMO}}$ introduced in \cite{FJN}. Particularly,
in \cite[Theorem 2.1]{FJN}, it was proved that, for all $\alpha<1$, $H^{\alpha,2}$ contains
the space $e^{-t\sqrt{-\Delta}}\mathcal{L}_{2,n+2\alpha}$ consisting of all the Poisson extensions
$$
e^{-t\sqrt{-\Delta}}f(x)=\pi^{-\frac{n+1}{2}}\Gamma\bigg(\frac{n+1}{2}\bigg)\int_{\mathbb R^n}(|x-y|^2+t^2)^{-\frac{n+1}2}t f(y)\,dy,
\ \ \ x\in \rr^n,
$$
of all $f$ in the square Campanato space $\mathcal{L}_{2,n+2\alpha}$ on $\mathbb R^n$ (cf. \cite{Cam1, Cam2, Sta}).
Recall that the Campanato space $\mathcal{L}_{2,n+2\alpha}\equiv\mathcal{L}_{2,n+2\alpha}(\rr^n)$ is the collection of all $L_{\mathrm{loc}}^2$ functions
satisfying
$$\|f\|_{{\mathcal L}_{2,n+2\alpha}}:=
\left(\displaystyle\sup_{(x_0,r)\in \mathbb R^{n+1}_+}r^{-(n+2\alpha)}\int_{B(x_0,r)} |f(x)-f_{B(x_0,r)}|^2\,dx\right)^\frac12<\infty,$$
where $f_{B(x_0,r)}$ denotes the integral average of $f$ on $B(x_0,r)$, i.\,e.,
$$f_{B(x_0,r)}:=\frac 1{|B(x_0,r)|}\int_{B(x_0,r)}f(y)\,dy.$$

Moreover, in \cite[Theorems 1.0 and 2.2]{FJN}, it was proved that,
if $\alpha\in [0,1)$, then $H^{\alpha,2}$ is contained in
$e^{-t\sqrt{-\Delta}}\mathcal{L}_{2,n+2\alpha}$ and hence $H^{\alpha,2}$ coincides with $e^{-t\sqrt{-\Delta}}\mathcal{L}_{2,n+2\alpha}$.
However, one point left unsolved is the characterization of $H^{\alpha,2}$ for $\alpha\in (-1,0)$.

The main aim of this article is to extend aforementioned characterizations of harmonic functions to the cases $\alpha\in (-1,1)$,
and to establish analogues for caloric functions. The main results of this article are Theorems \ref{t21}, \ref{t31}, \ref{t41} and \ref{t42}
below. Our first theorem below extends the known case $\alpha\in [0,1)$
established in \cite{FJN, FJN2} (see also \cite{FN, Tor} for more information) to $\alpha\in (-1,1)$.

\begin{thm}\label{t21}
If $\alpha\in (-1,1)$, then $H^{\alpha,2}=e^{-t\sqrt{-\Delta}}\mathcal{L}_{2,n+2\alpha}$ with equivalent norms.
\end{thm}


It then turns out, surprisingly but naturally, such an equation for both spaces can be transplanted to
their scaling invariant spaces.

The scaling invariant square Campanato-Sobolev space $(-\Delta)^{\frac{\alpha}{2}}\mathcal{L}_{2,n+2\alpha}\equiv (-\Delta)^{\frac{\alpha}{2}}\mathcal{L}_{2,n+2\alpha}(\rn)$,
the scaling invariant of $\mathcal{L}_{2,n+2\alpha}$, is the collection of all measurable functions $f$ such that $(-\Delta)^{-\frac{\alpha}{2}}f$ belong to the space $\mathcal{L}_{2,n+2\alpha}$,
 and we equip $f$ with the norm
$$
\lf\|f\r\|_{(-\Delta)^{\frac{\alpha}{2}}\mathcal{L}_{2,n+2\alpha}}:= \lf\|(-\Delta)^{-\frac{\alpha}{2}}f\r\|_{{{\mathcal L}_{2,n+2\alpha}}}.
$$
Here $(-\Delta)^{-\frac{\alpha}{2}} f$, determined through the Fourier transform
$$\widehat{(-\Delta)^{-\frac{\alpha}{2}}f}(x):=(2\pi|x|)^{-\alpha}\hat{f}(x),\quad \forall\,x\in\rn,$$
expresses the $(-\alpha/2)$-th power of the spatial Laplacian
$-\Delta f(x)$. 
Obviously, $(-\Delta)^{\frac{\alpha}{2}}\mathcal{L}_{2,n+2\alpha}$ is the ${\mathrm{BMO}}$ space when $\alpha=0$, while
$(-\Delta)^{\frac{\alpha}{2}}\mathcal{L}_{2,n+2\alpha}$ coincide with the space $Q_{-\alpha}$ introduced by \cite{EJPX}
(cf. \cite{Xiao1, Xiao2,WX}) for $\alpha\in (-1,0]$. See Section 3 below for more properties of the spaces
$(-\Delta)^{\frac{\alpha}{2}}\mathcal{L}_{2,n+2\alpha}$.

The scaling invariant version of $H^{\alpha,2}$, $\mathcal{H}^{\alpha,2}\equiv\mathcal{H}^{\alpha,2}(\rr^{n+1}_+)$, is defined to be the collection
of all harmonic functions $u$ on $\mathbb R^{n+1}_+$ satisfying
$$
\|u\|_{\mathcal{H}^{\alpha,2}}:=
\left(\sup_{(x_0,r)\in\mathbb R^{n+1}_+}r^{-(2\alpha+n)}\int_{B(x_0,r)}
\int_0^r|\nabla_{x,t} u(x,t)|^2\,t^{1+2\alpha} \,dt\,dx\right)^\frac12<\infty.
$$
The space $\mathcal{H}^{\alpha,2}$ is scaling invariant and, for $\alpha\in (0,1)$, it is easy to see that
$$\mathcal{H}^{-\alpha,2}\subseteq \mathcal{H}^{0,2}={\mathrm{HMO}}\subseteq \mathcal{H}^{\alpha,2};$$
see Lemma \ref{l32} below for details.

Theorem \ref{t21} holds true also in the scaling invariant version as follows.
\begin{thm}\label{t31}
For $\alpha\in (-1,1)$, one has

{\rm(i)} $\mathcal{H}^{\alpha,2}=e^{-t\sqrt{-\Delta}}(-\Delta)^{\frac{\alpha}{2}}\mathcal{L}_{2,n+2\alpha}$ with equivalent norms.

{\rm(ii)} If $\beta\in (0,1)$, then $\mathcal{H}^{\beta,2}=\mathop\mathrm{HB}\equiv \mathop\mathrm{HB}(\rr^{n+1}_+)$ -- the Bloch space of all
harmonic functions $u$ in $\mathbb R^{n+1}_+$ satisfying
$$
\|u\|_{\mathop\mathrm{HB}}:=\sup_{(x,t)\in\mathbb R^{n+1}_+}t|\nabla_{x,t}u(x,t)|<\infty,
$$
and hence $\mathop\mathrm{HB}=e^{-t\sqrt{-\Delta}}(-\Delta)^{\frac{\beta}{2}}\mathcal{L}_{2,n+2\beta}$ with equivalent norms.
\end{thm}

Our approach is essentially different from the Green-formula-based-argument carried
out in \cite{FJN2, FN} and, consequently, can be applied to explore a similar
description of caloric (temperature) functions, i.\,e., solutions to the heat equation on  $\rr^{n+1}_+$:
$$
(\Delta-\partial_t)u(x,t)=0,\quad\forall\,(x,t)\in\mathbb R^{n+1}_+.
$$

In accordance with \cite{FN}, for $\alpha\in (-1,1)$ let $T^{\alpha,2}\equiv T^{\alpha,2}(\rr^{n+1}_+)$,
respectively $\mathcal{T}^{\alpha,2}\equiv\mathcal{T}^{\alpha,2}(\rr^{n+1}_+)$, be the classes of all caloric functions $u(x,t)$ on $\mathbb R^{n+1}_+$,
equipped with the norm
$$ \|u\|_{T^{\alpha,2}}:=\left(\displaystyle\sup_{(x_0,r)\in \mathbb R^{n+1}_+}r^{-(2\alpha+n)}\int_{B(x_0,r)}\int_0^{r^2}|\nabla u(x,t)|^2\,dt\,dx\right)^\frac12<\infty,$$
respectively
$$\|u\|_{\mathcal{T}^{\alpha,2}}:=\left(\displaystyle\sup_{(x_0,r)\in \mathbb R^{n+1}_+}r^{-(2\alpha+n)}\int_{B(x_0,r)}\int_0^{r^2}|\nabla u(x,t)|^2t^{\alpha}\,dt\,dx\right)^\frac12<\infty.$$

Notice that, when $\alpha=0$,   $T^{\alpha,2}$ coincides with
$\mathcal{T}^{\alpha,2}$, and they are indeed the space ${\mathrm{TMO}}$ introduced by Fabes and Neri \cite{FN}.
Moreover, it was proved in \cite{FN} that ${\mathrm{TMO}}=e^{t\Delta} {\mathrm{BMO}}$.

The following theorem extends Theorems \ref{t21} and \ref{t31} to caloric functions, which
generalizes the characterization of ${\mathrm{TMO}}$ in \cite{FN} to $T^{\alpha,2}$ and $\mathcal{T}^{\alpha,2}$.

\begin{thm}\label{t41} For $\alpha\in (-1,1)$, one has

{\rm(i)} $T^{\alpha,2}=e^{t\Delta}\mathcal{L}_{2,n+2\alpha}$ with equivalent norms.

{\rm(ii)}  $\mathcal{T}^{\alpha,2}=e^{t\Delta}(-\Delta)^\frac{\alpha}{2}\mathcal{L}_{2,n+2\alpha}$ with equivalent norms.

{{\rm(iii)} If $\beta\in (0,1)$, then $\mathcal{T}^{\beta,2}=\mathop\mathrm{CB}\equiv \mathop\mathrm{CB}(\rr^{n+1}_+)$ -- the Bloch space of all caloric
functions $u$ in $\mathbb R^{n+1}_+$ satisfying
$$
\|u\|_{\mathop\mathrm{CB}}:=\sup_{(x,t)\in\mathbb R^{n+1}_+}\sqrt{t}|\nabla u(x,t)|<\infty,
$$
and hence $\mathop\mathrm{CB}=e^{t\Delta}(-\Delta)^\frac{\beta}{2}\mathcal{L}_{2,n+2\beta}$} with equivalent norms.
\end{thm}

Notice that, for $\alpha\in (-1,1)$, the spaces $(-\Delta)^{\frac\alpha 2}\mathcal{L}_{2,n+2\alpha}$ are scaling invariant
and they obey the inclusions
$$
 Q_{\beta}=(-\Delta)^{-\frac\beta 2}\mathcal{L}_{2,n-2\beta}\subseteq {\mathrm{BMO}}=(-\Delta)^{0}\mathcal{L}_{2,n}
 \subseteq (-\Delta)^\frac{\beta}{2} \mathcal{L}_{2,n+2\beta}$$
for $\beta \in (0,1)$;  see Lemma \ref{l31} below.
We next introduce the spaces $\big((-\Delta)^\frac{\alpha}{2}\mathcal{L}_{2,n+2\alpha}\big)^{-1}$ and consider the
solvability of the three-dimensional incompressible Navier-Stokes system with initial values in these spaces.

For $T\in (0,\infty]$ and $\alpha\in (-1,1)$, the space
$\big((-\Delta)^\frac{\alpha}{2}\mathcal{L}_{2,n+2\alpha}\big)_T^{-1}$ is defined as the collection
of all functions $f$ on $\mathbb R^{n}$ satisfying
$$
\|f\|_{n+2\alpha,-1,T}:=\left(\sup_{(x_0,r)\in\mathbb R^{n}\times(0,T)}r^{-(2\alpha+n)}{\int_0^{r^2}\int_{B(x_0,r)}|e^{t\Delta}f(y)|^2 t^{\alpha}\,dy\,dt}\right)^\frac12<\infty,
$$
and its associated space $X_{n+2\alpha,T}$ of all functions $u$ on $\mathbb R^{n+1}_+$ satisfying
\begin{align*}
\|{u}\|_{X_{n+2\alpha,T}}&:=\sup_{(x,t)\in\mathbb R^{n}\times(0,T]}\sqrt{t}|u(x,t)|\\
&\quad+\left(\sup_{(x_0,r)\in\mathbb R^n\times(0,T)}r^{-(2\alpha+n)}{\int_0^{r^2}\int_{B(x_0,r)}|{u}(y,t)|^2 t^{\alpha}\,dy\,dt}\right)^\frac12<\infty.
\end{align*}

For $\alpha\in (-1,0]$, from \cite{EJPX,WX,Xiao1}, it follows that
$$(-\Delta)^\frac{\alpha}{2}\mathcal{L}_{2,n+2\alpha}=Q_{-\alpha}$$
($Q_{0}={\mathrm{BMO}}$), hence the spaces $\big((-\Delta)^\frac{\alpha}{2}\mathcal{L}_{2,n+2\alpha}\big)_T^{-1}$
are just the spaces ${\mathrm{BMO}}^{-1}_T$ and $Q_{-\alpha,T}^{-1}$
considered in \cite{KT, Xiao1, Xiao2}. Interestingly, when $\alpha\in (0,1)$, the spaces $\big((-\Delta)^\frac{\alpha}{2}\mathcal{L}_{2,n+2\alpha}\big)_\infty^{-1}$ turn out to be the space
$\dot{B}^{-1,\infty}_\infty$; see Theorem \ref{t42} below. Thus, the spaces $(-\Delta)^{\frac{\alpha}{2}}\mathcal{L}_{2,n+2\alpha}$
naturally unify the spaces $Q^{-1}_{\alpha}$, ${\mathrm{BMO}}^{-1}$ and $\dot{B}^{-1,\infty}_\infty$,
which explicitly displays the intrinsic and subtle connections existing in theses spaces and hence
can be considered as one of the most significant contributions of this article.

Now we consider the three-dimensional incompressible Navier-Stokes
system with the pressure function $p\equiv p(x,t)$:
$$
(\hbox{3D-N-S})\quad\quad\left\{\begin{array}{rl}
(\Delta-\partial_t){\bf u}-{\bf u}\cdot\nabla{\bf u}-\nabla p=0\ \ &\hbox{on}\ \ \mathbb R^{3}\times(0,T);\\
\nabla\cdot{\bf u}=0\ \ &\hbox{on}\ \ \mathbb R^3;\\
{\bf u}(\cdot,0)={\bf a}(\cdot)\ \ &\hbox{on}\ \ \mathbb R^3;\\
\nabla\cdot{\bf a}=0\ \ &\hbox{on}\ \ \mathbb R^3.
\end{array}
\right.
$$
According to \cite{Kato, Kato1, GM, Tay}, for each ${\bf a}\in \big((-\Delta)^\frac{\alpha}{2}\mathcal{L}_{2,n+2\alpha}\big)_T^{-1}$
where $\alpha\in (-1,1)$, the equation has a unique mild solution ${\bf u}=(u_1,u_2,u_3)$, i.\,e., a solution to the integral equation:
$$
{\bf u}(x,t)=e^{t\Delta}{\bf a}(x)-\int_0^t e^{(t-s)\Delta}P\nabla\cdot({\bf u}\otimes{\bf u})\,ds,
$$
where
$$
\left\{\begin{array}{r@{}l}
e^{t\Delta}{\bf a}(x)=(e^{t\Delta}a_1(x),e^{t\Delta}a_2(x),e^{t\Delta}a_3(x));\\
P=\{P_{jk}\}_{j,k=1,2,3}=\{\delta_{jk}+R_jR_k\}_{j,k=1,2,3};\\
\delta_{jk}=\hbox{Kronecker\ symbol};\\
R_j=\partial_j(-\Delta)^{-\frac12}=\hbox{Riesz\ transform}.
\end{array}
\right.
$$
Using results from \cite{KT, Xiao1, Xiao2, BP} and the scaling invariant nature of
$(-\Delta)^{\frac{\alpha}{2}}\mathcal{L}_{2,n+2\alpha}$, it readily follows that
the solution ${\bf u}$ admits
$\|{\bf u}\|_{X_{3+2\alpha,T}}:=\sum_{j=1}^3\|{ u}_j\|_{X_{3+2\alpha,T}}<\infty$ for any small norm
$$\|{\bf a}\|_{3+2\alpha,-1,T}:=\sum_{j=1}^3\|a_j\|_{3+2\alpha,-1,T}$$
 when $\alpha\in (-1,0]$, however a norm inflation in finite time when $\alpha\in (0,1)$; see Corollary \ref{c41} below.

In summary, we have the following helpful structure table on the relations 
between our results and known ones, which clearly
indicates how the results in this article construct a complete bridge 
connecting known conclusions on $Q^{-1}_{\alpha}$, ${\mathrm{BMO}}^{-1}$ and $\dot{B}^{-1,\infty}_\infty$:
\begin{center}
\begin{tabular}{|p{2.71cm}|p{3.8cm}|p{2.5cm}|p{3.8cm}|}
\hline
& $\alpha\in (-1,0)$ &$\alpha=0$ & $\alpha\in (0,1)$  \\ \hline
$\mathcal{L}_{2,n+2\alpha}$ & Square Morrey \cite{Mor} & ${\mathrm{BMO}}$ \cite{JN} & ${\mathop\mathrm{Lip}}\,\alpha$ \cite{FJN, Mey, OT, Tor}\\ \hline
$(-\Delta)^{\frac{\alpha}{2}}\mathcal{L}_{2,n+2\alpha}$ & $Q_{-\alpha}$ \cite{Xiao1, Xiao2, WX} & ${\mathrm{BMO}}$ \cite{JN} & $(-\Delta)^\frac{\alpha}{2}{\mathop\mathrm{Lip}}\,\alpha$ \cite{Str1, Str2}\\ \hline
$H^{\alpha,2}$ & Theorem \ref{t21} & ${\mathrm{HMO}}$ \cite{FJN} & Theorem \ref{t21}/ \cite{FJN} \\ \hline
$\mathcal{H}^{\alpha,2}$ & Theorem \ref{t31}(i) & ${\mathrm{HMO}}$ \cite{FJN} & Harmonic Bloch \cite{RY} -Theorem \ref{t31}(ii)\\ \hline
$T^{\alpha,2}$ & Theorem \ref{t41}(i) & ${\mathrm{TMO}}$ \cite{FN} & Theorem \ref{t41}(i)\\ \hline
$\mathcal{T}^{\alpha,2}$ & Theorem \ref{t41}(ii) & ${\mathrm{TMO}}$ \cite{FN} & Caloric Bloch \cite{H} - Theorem \ref{t41}(iii)\\ \hline
$\big((-\Delta)^{\frac{\alpha}{2}}\mathcal{L}_{2,n+2\alpha}\big)_\infty^{-1}$& $Q_{-\alpha}^{-1}$ - Theorem \ref{t42}(i) / \cite{Xiao1, Xiao2} & ${\mathrm{BMO}}^{-1}$ \cite{KT} & $ \dot{B}^{-1,\infty}_\infty$ - Theorem \ref{t42}(ii) / \cite{BP}\\ \hline
$\hbox{3D-N-S}$& Well-posed - Corollary \ref{c41}(i)  /
\cite{Xiao1, Xiao2} & Well-posed \cite{KT} & Ill-posed-Corollary \ref{c41}(ii) / \cite{BP}\\ \hline
\end{tabular}
\end{center}

This article  is organized as follows. Theorem \ref{t21} is 
proved in Section 2, and Theorem \ref{t31} is proved in Section 3.
Theorem \ref{t41} and Theorem \ref{t42} will be proved in Section 4.

{\it Notation}. From now on, $U\lesssim V$ will stand for $U\le C V$ for a constant $C>0$ which is independent of the main parameters.
The symbol $U\sim V$ stands for $U\lesssim V\lesssim  U$.

\section{Both $\mathcal{L}_{2,n+2\alpha}$ and $H^{\alpha,2}$}\label{s2}

\subsection{The Campanato spaces $\mathcal{L}_{2,n+2\alpha}$}\label{s21}
\hskip\parindent We begin with giving a basic structure of the square Campanato space.

\begin{lem}\label{l21} Let $\alpha\in (-1,1)$.

{\rm(i)} For each measurable function $f$ on $\rr^n$, it holds true that
\begin{align*}
\|f\|_{{\mathcal L}_{2,n+2\alpha}}&\sim \lf(\sup_{(x_0,r)\in\mathbb R^{n+1}_+}r^{-2(\alpha+n)}\iint_{B(x_0,r)\times B(x_0,r)}|f(y)-f(z)|^2\,dy\,dz\r)^{1/2}\\
&\sim \lf(\sup_{(x_0,r)\in\mathbb R^{n+1}_+}r^{-(2\alpha+n)}\int_{B(x_0,r)}\int_0^r|\nabla_{x,t} e^{-t\sqrt{-\Delta}} f(x)|^2\,t\,dt\,dx\r)^{1/2}.
\end{align*}

{\rm(ii)} $\mathcal{L}_{2,n+2\alpha}$ is scaling variant except $\alpha=0$, i.\,e.,
$$
f\in\mathcal{L}_{2,n+2\alpha}\Longrightarrow \|f(\lambda\cdot)\|_{{\mathcal L}_{2,n+2\alpha}}=\lambda^{\alpha}\|f\|_{{\mathcal L}_{2,n+2\alpha}},\quad\forall \lambda>0.
$$

{\rm(iii)} There is no mutual inclusion among the family $\mathcal{L}_{2,n+2\alpha}$.
\end{lem}

\begin{proof} This follows from some simple computations, \cite[Theorem 1.2(i) \& Lemma 2.1]{Xiao1} for $\alpha\in (-1,0)$ and \cite[Theorems 2.1 and 2.2]{FJN} for $\alpha\in [0,1)$.
\end{proof}

\subsection{The harmonic spaces $H^{\alpha,2}$}
\hskip\parindent
The following lemma comes from a deep understanding of \cite[Lemmas 2.4 and 2.5]{FJN} and \cite[Lemmas 1.3 and 1.4]{FJN}.

\begin{lem}\label{l22} For $\alpha\in (-1,1)$, one has

{\rm(i)} If $u\in H^{\alpha,2}$, then $|\nabla_{x,t} u(x,t)|\lesssim t^{\alpha-1}\|u\|_{H^{\alpha,2}},\ \ \forall\,(x,t)\in\mathbb R^{n+1}_+$.

{\rm(ii)} If $u$ is a $C^1$ function obeying $|\nabla_{x,t} u(x,t)|\lesssim t^{\alpha-1}$, $\forall\,(x,t)\in\mathbb R^{n+1}_+$, then
$$
|u(x,t)-u(x_0,t)|\lesssim
\begin{cases}
t^{\alpha-1}|x-x_0|,\quad&\hbox{as}\quad |x-x_0|\le t;\\
\begin{cases}
\max\{1,|\alpha|^{-1}\}|x-x_0|^\alpha\ \ &\hbox{for}\  \alpha>0,\\
\log(t^{-1}|x-x_0|)\ \ &\hbox{for}\ \ \alpha=0,\\
\max\{1,|\alpha|^{-1}\}t^\alpha\ \ &\hbox{for}\  \alpha<0,
\end{cases}\quad&\hbox{as}\  |x-x_0|>t,
\end{cases}
$$
and hence $e^{-t\sqrt{-\Delta}} u(\cdot,k^{-1})(x)$ exists everywhere on $\mathbb R^{n+1}_+$ for any natural number $k$.

{\rm(iii)} If $u\in H^{\alpha,2}$, $u_k(x,t):= u(x,t+k^{-1})$ and $\lambda>0$, then
$$
\lf\{\begin{array}{rl}
u_k(x,t)&=e^{-t\sqrt{-\Delta}} u(\cdot,k^{-1})(x);\\
\|u_k\|_{H^{\alpha,2}}&\lesssim \|u\|_{H^{\alpha,2}};\\
\|u(\lambda\cdot,\lambda\cdot)\|_{H^{\alpha,2}}&=\lambda^{2(\alpha-1)}\|u\|_{H^{\alpha,2}},
\end{array}\r.
$$
and hence there is no mutual inclusion among the family $H^{\alpha,2}$.
\end{lem}

\begin{proof} Only (iii) needs a verification. Notice that $u_k(x,t)$ and $e^{-t\sqrt{-\Delta}} u(\cdot,k^{-1})(x)$ tend to $u(x,k^{-1})$ as $t\to 0$ and $|\nabla_{x,t} u_k(x,t)|\lesssim (t+k^{-1})^{\alpha-1}$ (due to (i)-(ii) above). So, $\partial_j u(x,t+k^{-1})$ and $\partial_t u(x,t+k^{-1})$ are harmonic on $\mathbb R^{n+1}_+$ for each $j\in\{1,2,...,n\}$ and take $\partial_j u(x,k^{-1})$ and $\partial_t u(x,k^{-1})$ as their pointwise limits as $t\to 0$. An application of the maximum principle for harmonic functions yields
$$
\nabla_{x,t} u(x,t+k^{-1})=e^{-t\sqrt{-\Delta}}\nabla_{x,t} u(\cdot,k^{-1})(x)=\nabla_{x,t} e^{-t\sqrt{-\Delta}} u(\cdot,k^{-1})(x).
$$
Thus, there is a constant $c_k$ such that
$$
u(x,t+k^{-1})-e^{-t\sqrt{-\Delta}} u(\cdot,k^{-1})(x)=c_k.
$$
Letting $t\to 0$ in the last equation derives $c_k=0$. So, the desired equality follows.

To show the estimate in (iii), we consider two cases.

{\it Case 1: $kr\ge 1$.} Using $s=t+k^{-1}$, we see that
\begin{eqnarray*}
&&r^{-(2\alpha+n)}\int_{B(x_0,r)}\int_0^r |\nabla_{x,t} u_k(x,t)|^2\,t\,dt\,dx\\
&&\quad\lesssim r^{-(2\alpha+n)}
\int_{B(x_0,r)}\int_0^{2r}|\nabla_{x,t} u(x,s)|^2\,s\,ds\,dx\lesssim\|u\|_{H^{\alpha,2}}^2.
\end{eqnarray*}

{\it Case 2: $kr<1$}. Using (i), we find that
\begin{align*}
r^{-(2\alpha+n)}\int_{B(x_0,r)}\int_0^r |\nabla_{x,t} u_k(x,t)|^2\,t\,dt\,dx
&\lesssim \|u\|^2_{H^{\alpha,2}}r^{-(2\alpha+n)}\int_{B(x_0,r)}\int_0^r t(t+k^{-1})^{2(\alpha-1)}\,dt\,dx\\
&\lesssim\|u\|^2_{H^{\alpha,2}}r^{-2\alpha}k^{2(1-\alpha)}\int_0^r t(kt+1)^{2(\alpha-1)}dt\\
&\lesssim \|u\|^2_{H^{\alpha,2}}(kr)^{2(1-\alpha)}\\
&\lesssim \|u\|^2_{H^{\alpha,2}}.
\end{align*}

Putting the above two cases together, one obtains
$\|u_k\|^2_{H^{\alpha,2}}\lesssim \|u\|^2_{H^{\alpha,2}},$
as desired.

The scaling formula follows directly from changing variables: $(x,t)\to (\lambda x,\lambda t)$.
\end{proof}

We are now in  a position to prove Theorem \ref{t21}. Notice that our arguments below work for all $\alpha\in (-1,1)$.

\begin{proof}[Proof of Theroem \ref{t21}]
The inclusion $H^{\alpha,2}\supseteq e^{-t\sqrt{-\Delta}}\mathcal{L}_{2,n+2\alpha}$ follows
from Lemma \ref{l21}(i). So, it remains to prove the inclusion $H^{\alpha,2}\subseteq e^{-t\sqrt{-\Delta}}\mathcal{L}_{2,n+2\alpha}$. To do so, suppose $u\in H^{\alpha,2}$. In order to obtain a function $f\in \mathcal{L}_{2,n+2\alpha}$ solving the convolution integral equation:
 $$
 u(x,t)=e^{-t\sqrt{-\Delta}} f(x),\quad\forall\,(x,t)\in \mathbb R^{n+1}_+, $$
 let $f_k(x):= u(x,k^{-1})$ for each $k\in\mathbb{N}:=\{1,2,\cdots\}$. Then three steps are required to complete our argument.

{\it Step 1} - establishing  $\sup_k\|f_k\|_{{\mathcal L}_{2,n+2\alpha}}\lesssim\|u\|_{H^{\alpha,2}}$.

By Lemma \ref{l22}(iii), we see that
$$
\sup_{(x_0,r)\in\mathbb R^{n+1}_+}r^{-(2\alpha+n)}\int_{B(x_0,r)}\int_0^r|\nabla_{x,t} e^{-t\sqrt{-\Delta}} f_k(x)|^2t\,dt\,dx\lesssim\|u\|^2_{H^{\alpha,2}}
$$
whence reaching the desired estimate via Lemma \ref{l21}(i).

{\it Step 2} - finding a function $f\in\mathcal{L}_{2,n+2\alpha}$ through $L^2(B(0,2^m))$-boundedness of $\{f_k-(f_k)_{B(0,2^m)}\}$ for each $m\in \mathbb{N}$.

According to {\it Step 1}, for a given natural number $m$, one has
$$
\int_{B(0,2^m)}|f_k(x)-(f_k)_{B(0,2^m)}|^2\,dx\lesssim\|u\|^2_{H^{\alpha,2}}2^{m(2\alpha+n)},\ \ \forall\, k\in \mathbb{N}.
$$
Namely, the sequence $\{f_k-(f_k)_{B(0,2^m)}\}_{k=1}^\infty$ is bounded in $L^2(B(0,2^m))$. So, after passing to a subsequence, the sequence converges weakly to a function $g_m\in L^2(B(0,2^m))$ with
$$
(g_m)_{B(0,2^m)}=\lim_{k\to\infty}\big(f_k-(f_k)_{B(0,2^m)}\big)_{B(0,2^m)}=0.
$$
Defining 
$$\delta\big(B(0,2),B(0,2^m)\big):= \lim_{k\to\infty}\lf((f_k)_{B(0,2)}-(f_k)_{B(0,2^m)}\r)$$ 
and, for all $x\in B(0,2^m)$, 
$$f(x):= g_m(x)-\delta\big(B(0,2),B(0,2^m)\big),$$
we readily see that $f$ is well defined on $\mathbb R^n=\cup_{m=1}^\infty B(0,2^m)$.

Notice that $f_1\in\mathcal{L}_{2,n+2\alpha}$. So, $f\in\mathcal{L}_{2,n+2\alpha}$  follows from checking $f-f_1\in\mathcal{L}_{2,n+2\alpha}$. To see this, fixing any open ball $B\subseteq\mathbb R^n$, choosing $B(0,2^m)$ to cover $B$, letting $g^B$ be the $L^2$-weak limit of $\{f_k-(f_k)_B\}_{k=1}^\infty$ on $B$, and using $g^B\in L^2(B)$, we find that
$$
g^B-g_m+\delta\big(B(0,2),B(0,2^m)\big)+f_B=0
$$
and, by the H\"older inequality, we further conclude that
\begin{align*}
\int_{B}|g^B(x)|^2\,dx&=\lim_{k\to\infty}\int_{B}g^B(x)(f_k(x)-(f_k)_B)\,dx\\
&\le\bigg[\int_{B}|g^B(x)|^2\,dx\bigg]^\frac12\limsup_{k\to\infty}\bigg[\int_{B}|f_k(x)-(f_k)_B|^2\,dx\bigg]^\frac12\\
&\lesssim |B|^\frac{n+2\alpha}{2n}\bigg[\int_{B}|g^B(x)|^2\,dx\bigg]^\frac12\|u\|_{H^{\alpha,2}}
\end{align*}
and hence
\begin{align*}
\int_{B}|f_1(x)-f(x)-(f_1)_B+f_B|^2\,dx&=\int_{B}|f_1(x)-(f_1)_B-g^B(x)|^2\,dx\\
&\lesssim \int_{B}|f_1(x)-(f_1)_B|^2\,dx+\int_B |g^B(x)|^2\,dx\\
&\lesssim |B|^\frac{n+2\alpha}{n}\big(\|f_1\|_{{\mathcal L}_{2,n+2\alpha}}^2+\|u\|^2_{H^{\alpha,2}}\big)\\
&\lesssim |B|^\frac{n+2\alpha}{n}\|u\|^2_{H^{\alpha,2}}.
\end{align*}
Consequently, $f_1-f\in \mathcal{L}_{2,n+2\alpha}$.

{\it Step 3} - verifying $u(x,t)=e^{-t\sqrt{-\Delta}} (f+c)(x)$ for some constant $c$.

Given $(x,t)\in\mathbb R^{n+1}_+$ and the standard Poisson kernel
$$
e^{-t\sqrt{-\Delta}}(x,y)=\Gamma\lf(\frac{n+1}{2}\r)\pi^{-\frac{n+1}{2}}t(|x-y|^2+t^2)^{-\frac{n+1}{2}},
$$
using the argument for Lemma \ref{l22}(iii) and the fact that
$\nabla_{x,t}e^{-t\sqrt{-\Delta}}(x,0)$ has integral zero on $\mathbb R^n$, we obtain
\begin{align*}
\nabla_{x,t} u(x,t+k^{-1})&=\int_{\mathbb R^n}\nabla_{x,t} e^{-t\sqrt{-\Delta}}(x,y)\big(f_k(y)-(f_k)_{B(0,2^m)}\big)\,dy\\
&= M_{k,m}(x,t)+N_{k,m}(x,t),
\end{align*}
where
$$
\begin{cases}
M_{k,m}(x,t):=\displaystyle\int_{B(0,2^m)}\nabla_{x,t} e^{-t\sqrt{-\Delta}}(x,y)\big(f_k(y)-(f_k)_{B(0,2^m)}\big)\,dy;\\
N_{k,m}(x,t):=\displaystyle\int_{\mathbb R^n\setminus B(0,2^m)}\nabla_{x,t} e^{-t\sqrt{-\Delta}}(x,y)\big(f_k(y)-(f_k)_{B(0,2^m)}\big)\,dy.
\end{cases}
$$
Notice that
$$
|\delta\big(B(0,2),B(0,2^m)\big)|\le|f_{B(0,2)}|+|f_{B(0,2^m)}-f_{B(0,2)}|\lesssim |f_{B(0,2)}|+m 2^{m\alpha}\|f\|_{{\mathcal L}_{2,n+2\alpha}}
$$
and
\begin{align*}
\Big|\int_{B(0,2^m)}\nabla_{x,t} e^{-t\sqrt{-\Delta}}(x,y)\,dy\Big|&=\Big|\int_{\mathbb R^n}\nabla_{x,t} e^{-t\sqrt{-\Delta}}(x,y)\,dy-
\int_{\mathbb R^n\setminus B(0,2^m)}\nabla_{x,t} e^{-t\sqrt{-\Delta}}(x,y)\,dy\Big|\\
&=\Big|\int_{\mathbb R^n\setminus B(0,2^m)}\nabla_{x,t} e^{-t\sqrt{-\Delta}}(x,y)\,dy\Big|\\
&\lesssim\int_{\mathbb R^n\setminus B(0,2^m)}|y-x|^{-n-1}\,dy\\
&\lesssim 2^{-m}.
\end{align*}
So, one has that, under $\alpha\in (-1,1)$,
\begin{align*}
\lim_{k,m\to\infty}M_{k,m}(x,t)
&=\lim_{m\to\infty}\lf[\int_{B(0,2^m)}\nabla_{x,t}
e^{-t\sqrt{-\Delta}}(x,y)f(y)\,dy\r.\\
&\hspace{0.6cm}\lf.+\delta\big(B(0,2),B(0,2^m)\big)\int_{B(0,2^m)}\nabla_{x,t} e^{-t\sqrt{-\Delta}}(x,y)\,dy\r]\\
&=\int_{\mathbb R^n}\nabla_{x,t} e^{-t\sqrt{-\Delta}}(x,y)f(y)\,dy.
\end{align*}

At the same time, noticing that $|x-y|\approx |y|$ for $y\in \mathbb R^n\setminus B(0,2^m)$ and the sufficiently large $m$, one finds
that
\begin{align*}
N_{k,m}(x,t)&\le\int_{\mathbb R^n\setminus B(0,2^m)}|\nabla_{x,t} e^{-t\sqrt{-\Delta}}(x,y)||f_k(y)-(f_k)_{B(0,2^m)}|\,dy\\
&\lesssim \int_{\mathbb R^n\setminus B(0,2^m)}|y|^{-n-1}|f_k(y)-(f_k)_{B(0,2^m)}|\,dy\\
&\lesssim \sum_{j=m+1}^\infty 2^{-j(n+1)}\int_{B(0,2^j)}|f_k(y)-(f_k)_{B(0,2^j)}+(f_k)_{B(0,2^j)}- (f_k)_{B(0,2^m)}|\,dy\\
&\lesssim\|u\|_{H^{\alpha,2}}\sum_{j=m+1}^\infty 2^{j(\alpha-1)}\big[1+(j-m)2^{-\alpha(j-m)}\big],
\end{align*}
whence deriving $\lim_{k,m\to\infty}N_{k,m}(x,t)=0$.

The above limits on $M_{k,m}(x,t)$ and $N_{k,m}(x,t)$ in turn give
$$
\nabla_{x,t} u(x,t)=\lim_{k\to\infty}\nabla_{x,t} u(x,t+k^{-1})=\nabla_{x,t} e^{-t\sqrt{-\Delta}} f(x).
$$
Hence, $u(x,t)=e^{-t\sqrt{-\Delta}} (f+c)(x)$ for some constant $c$, which finishes the proof of Theorem \ref{t21}.
\end{proof}

\begin{rem}\label{r21} \rm Under $\alpha\in (-1,0]$, a duality-based argument for $H^{\alpha,2}\subseteq e^{-t\sqrt{-\Delta}}\mathcal{L}_{2,n+2\alpha}$ can be provided here; see also \cite{FJN, FN, DYZ} for some related treatments. To do so,
let $H^{1,\alpha}_2$ be the set of
all functions $f\in L^1$ on $\mathbb R^n$ that satisfy
$\|f\|_{H^{1,\alpha}_2}:=\|{\mathcal A}_{2,\alpha}(f)\|_{L^1}<\infty,$
where
$$
{\mathcal A}_{2,\alpha}(f)(z):=
\lf(\int_0^\infty \int_{B(z,t)}|t\partial_t e^{-t\sqrt{-\Delta}} f(y)|^2\,{t^{2\alpha-n-1}}\,dy\,dt\r)^{\frac12},\quad\forall\, z\in\mathbb R^n.
$$
Then the dual space of $H^{1,\alpha}_2$ is $\mathcal{L}_{2,n+2\alpha}$; see also \cite{DXY, Xu, SXY}.
Importantly, each component of the vector $\nabla_{x,t}e^{-t\sqrt{-\Delta}}(x,0)$ is an element of $H^{1,\alpha}_2$. To see this, one utilizes
$$
e^{-(t+s)\sqrt{-\Delta}}(x,0)=\pi^{-\frac{n+1}{2}}\Gamma\lf(\frac{n+1}{2}\r)({t+s})
{\lf(|t+s|^2+|x|^2\r)^{-\frac{n+1}{2}}},\quad\forall\,s,t>0\ \ \&\, x\in\mathbb R^n,
$$
to derive
$$
\lf|t\partial_t e^{-t\sqrt{-\Delta}}\partial_{j}e^{-s\sqrt{-\Delta}}(x,0)\r|=\left|t\partial_{j}\partial_t e^{-(t+s)\sqrt{-\Delta}}(x,0)\r|\\
\lesssim{t}{\lf(|t+s|^2+|x|^2\r)^{-\frac{n+2}{2}}}.
$$
This implies that, under $\alpha\in (-1,0]$, one has
\begin{align*}
\lf[{\mathcal A}_{2,\alpha}(\partial_je^{-s\sqrt{-\Delta}}(x,0))(z)\r]^2&=
\int_0^\infty \int_{B(z,t)}|t\partial_t e^{-t\sqrt{-\Delta}}\partial_je^{-s\sqrt{-\Delta}}(x,0)|^2{t^{2\alpha-n-1}}\,dx\,dt\\
&\lesssim \int_0^\infty \int_{B(z,t)}{t^{1+2\alpha-n}}{\lf(|t+s|^2+|x|^2\r)^{-(n+2)}}{\,dx\,dt}\\
&\lesssim \int_0^\infty \int_{B(z,t)}{t^{1+2\alpha-n}}{\lf(|t+s|+|z|\r)^{-2(n+2)}}{\,dy\,dt}\\
&\lesssim {\lf(s+|z|\r)^{2(\alpha-n-1)}}.
\end{align*}
Meanwhile, by using the semigroup property, one derives
$$
\lf|t\partial_t e^{-t\sqrt{-\Delta}}\partial_s
e^{-s\sqrt{-\Delta}}(x,0)\r|=\left|t\partial_s\partial_t e^{-(t+s)\sqrt{-\Delta}}(x,0)\r|\lesssim{t}{\lf(|t+s|^2+|x|^2\r)^{-\frac{n+2}{2}}},
$$
whence reaching
$$
\big[{\mathcal A}_{2,\alpha}(\partial_s e^{-s\sqrt{-\Delta}}(x,0))(z)\big]^2\lesssim {\lf(s+|z|\r)^{2(\alpha-n-1)}}.
$$
The above estimates, via an $L^1$-estimate on $\mathbb R^n$, imply that each component of $\nabla_{x,t} e^{-t\sqrt{-\Delta}}(x,0)$ belongs to ${H^{1,\alpha}_2}$. Now, an application of {\it Step 1} deriving that $f_k$, up to passing a subsequence, is weakly convergent to a function $f\in \mathcal{L}_{2,n+2\alpha}$ with $\|f\|_{{\mathcal L}_{2,n+2\alpha}}\lesssim\|u\|_{H^{\alpha,2}}$.
Then, by using the duality $(H^{1,\alpha}_2)^\ast=\mathcal{L}_{2,n+2\alpha}$ and the fact that $u$ is harmonic in $\mathbb R^{n+1}_+$,
one has
\begin{align*}
\nabla_{x,t} u(x,t)&=\lim_{k\to \infty}\nabla_{x,t} u(x,t+k^{-1})\\
&=\lim_{k\to \infty}(\nabla_{x,t} e^{-t\sqrt{-\Delta}}(x,0))\ast f_k(x)\\
&=(\nabla_{x,t}e^{-t\sqrt{-\Delta}}(x,0))\ast f(x)\\
&=\nabla_{x,t} e^{-t\sqrt{-\Delta}} f(x),
\end{align*}
for which the last two $\ast$'s stand for the usual convolution. So, $u(x,t)=e^{-t\sqrt{-\Delta}} (f+c)(x)$ holds true for some constant $c$, as desired.
\end{rem}

\section{Both $(-\Delta)^\frac{\alpha}{2}\mathcal{L}_{2,n+2\alpha}$ and $\mathcal{H}^{\alpha,2}$}
\label{s3}

\subsection{The Campanato-Sobolev spaces $(-\Delta)^{\frac{\alpha}{2}}\mathcal{L}_{2,n+2\alpha}$}\label{s31}
\hskip\parindent The following lemma indicates that $(-\Delta)^{\frac{\alpha}{2}}\mathcal{L}_{2,n+2\alpha}$ exists as the scaling invariant version of $\mathcal{L}_{2,n+2\alpha}$.

 For $\beta\in [0,1)$, let $H^{\frac{n}{n+\beta}}\equiv H^{\frac{n}{n+\beta}}(\rn)$ be the Hardy space on $\rn$ introduced by \cite{FS}. We then define the space $(-\Delta)^{-\frac{\beta}{2}} H^{\frac{n}{n+\beta}}$ as the collection of distributions $f\in \mathcal S'$  with the norm
 $$\lf\|f\r\|_{(-\Delta)^{-\frac{\beta}{2}} H^{\frac{n}{n+\beta}}}:=\lf\|(-\Delta)^{\frac{\beta}{2}}f\r\|_{H^{\frac{n}{n+\beta}}}<\infty;$$
 see \cite{N} for more properties of such spaces. Here $\mathcal S'\equiv \mathcal S'(\rn)$ denotes the space of all Schwartz distributions
 on $\rn$.
 It readily follows, from the mapping property of the Riesz potential, that  $(-\Delta)^{-\frac{\beta}{2}} H^{\frac{n}{n+\beta}}\subset H^1$.
\begin{lem}
\label{l31}
Let $\alpha\in (-1,1)$.

{\rm(i)} For each measurable function $f$ on $\rr^n$, it holds true that
\begin{align*}
\|f\|_{(-\Delta)^{\frac{\alpha}{2}}\mathcal{L}_{2,n+2\alpha}}
&=\lf\|(-\Delta)^{-\frac{\alpha}{2}}f\r\|_{\mathcal{L}_{2,n+2\alpha}}\\
&\sim
\|f\|_{\mathcal{T}^{\alpha,2}}:=\left(\sup_{(x_0,r)\in\mathbb R^{n+1}_+}
\int_{B(x_0,r)}\int_0^r |\nabla_{x,t} e^{-t\sqrt{-\Delta}}(-\Delta)^{-\frac{\alpha}{2}}f(x)|^2\frac{t\,dt\,dx}{r^{2\alpha+n}}\right)^\frac12\\
&\sim\lf(\sup_{(x_0,r)\in\mathbb R^{n+1}_+}r^{-(2\alpha+n)}
\int_{B(x_0,r)}\int_0^r |\nabla_{x,t} (-\Delta)^{-\frac{\alpha}{2}}e^{-t\sqrt{-\Delta}} f(x)|^2{t\,dt\,dx}\r)^{\frac 12}.
\end{align*}

{\rm(ii)} $(-\Delta)^\frac{\alpha}{2}\mathcal{L}_{2,n+2\alpha}$ is scaling invariant, i.\,e.,
$$
f\in(-\Delta)^\frac{\alpha}{2}\mathcal{L}_{2,n+2\alpha}\Longrightarrow
\|f(\lambda \cdot)\|_{\mathcal{T}^{\alpha,2}}=\|f\|_{\mathcal{T}^{\alpha,2}},\ \ \forall\,\lambda>0.
$$

{\rm(iii)} For $\beta\in (0,1)$, one has
$$
(-\Delta)^{-\frac12}\mathcal{L}_{2,n-2}\subseteq Q_{\beta}\subseteq {\mathrm{BMO}}\subseteq (-\Delta)^\frac{\beta}{2} {\mathop\mathrm{Lip}}\,\beta= \lf((-\Delta)^{-\frac{\beta}{2}} H^{\frac{n}{n+\beta}}\r)^\ast.
$$
\end{lem}

\begin{proof} (i) This is a straightforward consequence of the definition, the Fourier transform and Lemma \ref{l21}(i).

(ii) This follows directly from a simple calculation with the norm.

(iii) Noticing (cf. \cite{EJPX, Xiao1, KXZZ}) that,
for $\beta\in (0,1)$, it holds true that
\begin{eqnarray}\label{32}
\|f\|_{Q_\beta}^2&\sim&\sup_{(x_0,r)\in\mathbb R^{n+1}_+}r^{2\beta-n}\iint_{B(x_0,r)\times B(x_0,r)}{|f(x)-f(y)|^2}{|x-y|^{-(n+2\beta)}}\,dx\,dy\\
&\sim&\sup_{(x_0,r)\in\mathbb R^{n+1}_+}r^{2\beta-n}\int_{B(x_0,r)}\int_0^r |\nabla_{x,t} e^{-t\sqrt{-\Delta}} f(x)|^2 t^{1-2\beta}\,dx\,dt\nonumber\\
&\sim& \|f\|_{(-\Delta)^{-{\beta}/{2}} \mathcal{L}_{2,n-2\beta}}^2\nonumber
\end{eqnarray}
and  (cf. \cite{FS, Str1, Str2, FJN})
\begin{eqnarray*}
\|f\|_{{\mathrm{BMO}}}^2&\sim&\sup_{(x_0,r)\in\mathbb R^{n+1}_+}r^{-2n}\iint_{B(x_0,r)\times B(x_0,r)}{|f(x)-f(y)|^2}\,dx\,dy\\
&\sim&\sup_{(x_0,r)\in\mathbb R^{n+1}_+}r^{-n}\int_{B(x_0,r)}\int_0^r |\nabla_{x,t} e^{-t\sqrt{-\Delta}} f(x)|^2 t\,dxdx,
\end{eqnarray*}
we readily show the second inclusion of (iii). Below is the argument for the
first one that improves the case $p=2$ of \cite[Corollary (i)]{Adams}.
Suppose $f\in (-\Delta)^{-\frac12}\mathcal{L}_{2,n-2}$ and let $g=(-\Delta)^{\frac12}f\in
\mathcal{L}_{2,n-2}$. Notice that the proof of \cite[Lemma 2.1]{Xiao1} actually reveals
\begin{align*}
&\sup_{(x_0,r)\in\mathbb R^{n+1}_+}r^{2\beta-n}\iint_{B(x_0,r)\times(0,r)}|F(x,t)|^2t^{-(1+2\beta)}\,dt\,dx\\
&\quad\quad\le
\sup_{(x_0,r)\in\mathbb R^{n+1}_+}r^{2-n}\iint_{B(x_0,r)\times(0,r)}|F(x,t)|^2\,t^{-3}{\,dt\,dx}\\
&\quad\quad\lesssim \|g\|^2_{\mathcal{L}_{2,n-2}},
\end{align*}
where
$$
\lf\{\begin{array}{rl}
F(x,t)&=t^2\partial_te^{-t\sqrt{-\Delta}} f(x);\\
\widehat{F(\cdot,t)}(x)&=-t^3|x|\hat{g}(x)\exp(-t|x|).
\end{array}\r.
$$
So, according to \cite[Theorem 7.0(i)]{DX1}, one has
$$
x\mapsto\Pi_{\psi_0}(x)=\iint_{\mathbb R^{n+1}_{+}}F(y,t)(\psi_0)_t(x-y)\,t^{-1}\,dt\,dy
$$
belongs to $Q_\beta$, where
$$
\lf\{\begin{array}{rl}
(\psi_0)_t(z)&=t^{-n}\psi_0(z/t);\\
\displaystyle\psi_0(z)&=(1+|z|^2)^{-\frac{n+3}2}\lf[1+|z|^2-(n+1)\Gamma\lf(\frac{n+1}{2}\r)\pi^{-\frac{n+1}{2}}\r].
\end{array}\r.
$$
Since
$$
\widehat{\Pi_{\psi_0}F}(x)=2^{-1}\pi\widehat{(-\Delta)^{-\frac{1}{2}}g}(x), \quad \forall\,x\in\rn,
$$
one concludes that $f=(-\Delta)^{-\frac12}g$ is in $Q_\beta$.

The third inclusion in (iii) follows from $(-\Delta)^{-\frac{\beta}{2}}{\mathrm{BMO}}\subseteq {\mathop\mathrm{Lip}}\,\beta$ (cf. \cite[Theorem 3.4]{Str1}). And, the result that $(-\Delta)^\frac{\beta}{2} {\mathop\mathrm{Lip}}\,\beta$ is equal to the dual space of $(-\Delta)^{-\frac{\beta}{2}} H^{\frac{n}{n+\beta}}$ follows from the well-known fact that the predual space of  ${\mathop\mathrm{Lip}}\,\beta$ is identified with the Hardy space $H^{\frac{n}{n+\beta}}$ (see e.g. \cite{BF}) and the usual dual pairing $\langle f,g\rangle=\langle \hat{f},\hat{g}\rangle$; see \cite[p. 203, (5)]{LLo} or \cite{N}.
\end{proof}

\subsection{The scaling invariant harmonic spaces $\mathcal{H}^{\alpha,2}$}\label{s32}
\hskip\parindent
We begin with introducing a new norm, which will be proved to be equivalent to $\|\cdot\|_{\mathcal{H}^{\alpha,2}}$.
For each harmonic functions $u$ on $\rr^{n+1}_+$, let
 $$
 \|u\|_{\alpha,2,\star}:=\left(\sup_{(x_0,r)\in\mathbb R^{n+1}_+}r^{-(2\alpha+n)}\int_{B(x_0,r)}\int_0^r|\nabla_{x,t}(-\Delta)^{-\alpha/2} u(x,t)|^2t\,dt\,dx\right)^\frac12.
$$

The following lemma corresponds to Lemma \ref{l31}.

\begin{lem}\label{l32} For $\alpha\in (-1,1)$, one has

{\rm(i)} If $u$ is a harmonic function on $\rr^{n+1}_+$ satisfying $ \|u\|_{\alpha,2,\star}<\infty$, then
$$
|\nabla_{x,t}(-\Delta)^{-\frac{\alpha}{2}}u(x,t)|\lesssim t^{\alpha-1}\|u\|_{\alpha,2,\star},\quad\forall\,(x,t)\in\mathbb R^{n+1}_+.
$$

{\rm(ii)} $\mathcal{H}^{\alpha,2}$ is scaling invariant, i.\,e.,
$$
u\in \mathcal{H}^{\alpha,2}\Longrightarrow\|u(\lambda\cdot,\lambda\cdot)\|_{{\mathcal{H}^{\alpha,2}}}=\|u\|_{\mathcal{H}^{\alpha,2}},\quad\forall\,\lambda>0.
$$

{\rm(iii)} If $\beta\in (0,1)$, then
$$
\mathcal{H}^{-1,2}\subseteq \mathcal{H}^{-\beta,2}\subseteq \mathcal{H}^{0,2}={\mathrm{HMO}}\subseteq \mathcal{H}^{\beta,2}.
$$
\end{lem}
\begin{proof} (i) Notice that $u$ is infinitely differentiable and hence (see e.g. \cite{NY})
$$
(-\Delta)^{-\frac{\alpha}{2}}u(x,t)=
\begin{cases}
\displaystyle\frac{\alpha \Gamma(({n-\alpha})/{2})}{2^{\alpha+1}\pi^{n/2}\Gamma(1+\alpha/2)}
{\lim_{\epsilon\to 0}}\int_{\mathbb R^{n}\setminus B(0,\epsilon)}\frac{u(x,t)-u(x
+y,t)}{|y|^{n-\alpha}}\,dy\quad &\hbox{as}\quad  \alpha\in (-1,0);\\
u(x,t)\quad &\hbox{as}\quad \alpha=0;\\
\displaystyle\frac 1{2\pi\Gamma(\alpha)}\int_0^\infty u(x,t+s)\,s^{\alpha-1}ds\quad &\hbox{as}\quad \alpha\in (0,1),
\end{cases}
$$
and so $\Delta_{x,t}\big((-\Delta)^{-\frac{\alpha}{2}}u(x,t)\big)=0$, namely, $(-\Delta)^{-\frac{\alpha}{2}}u(x,t)$ is harmonic in $\mathbb R^{n+1}_+$.

As in the proof of  \cite[Lemma 1.1]{FJN}, an application of the mean-value inequalities for the subharmonic functions $|\partial_j (-\Delta)^{-\frac{\alpha}{2}}u(x,t)|^2$ and $|\partial_t (-\Delta)^{-\frac{\alpha}{2}}u(x,t)|^2$ on a given ball $B_{x_0,r}^{n+1}$ in $\mathbb R^{n+1}$ with centre $(x_0,r)$ and radius $r/2>0$ derives
$$
|\nabla_{x,t} (-\Delta)^{-\frac{\alpha}{2}}u(x,t)|^2\Big|_{(x,t)=(x_0,r)}\lesssim
r^{-n-1}\iint_{B_{x_0,r}^{n+1}}|\nabla_{x,t} (-\Delta)^{-\frac{\alpha}{2}}u(x,t)|^2\,dx\,dt\lesssim
r^{2(\alpha-1)}\|u\|^2_{\alpha,2,\star}.
$$

(ii) This follows from changing variables $(x,t)\to (\lambda x,\lambda t)$.

(iii)
This sequence of inclusions follows from the definition by noticing that
\begin{eqnarray*}
&\displaystyle\sup_{(x_0,r)\in\mathbb R^{n+1}_+}r^{-(2\alpha+n)}\int_{B(x_0,r)}
\int_0^r|\nabla_{x,t} u(x,t)|^2\,t^{1+2\alpha} \,dt\,dx\\
&\qquad \le \displaystyle\sup_{(x_0,r)\in\mathbb R^{n+1}_+}r^{-(2\beta+n)}\int_{B(x_0,r)}
\int_0^r|\nabla_{x,t} u(x,t)|^2\,t^{1+2\beta} \,dt\,dx
\end{eqnarray*}
holds true for all $-1\le\beta<\alpha<1$.
\end{proof}

\begin{lem}\label{l33}
Let $\alpha\in (-1,1)$. For all harmonic functions $u$ on $\rr^{n+1}_+$, it holds true that
 $$
\|u\|_{\mathcal{H}^{\alpha,2}}\sim\|u\|_{\alpha,2,\star}.
$$
\end{lem}
\begin{proof} Suppose that $u$ is a harmonic function on $\rr^{n+1}_+$.
Let us first show that $\|u\|_{\mathcal{H}^{\alpha,2}}\lesssim\|u\|_{\alpha,2,\star}$.

Notice that by the proof of Lemma \ref{l32}(i), $(-\Delta)^{-\frac{\alpha}{2}}u(x,t)$ is harmonic in $\mathbb R^{n+1}_+$.
Now let $v(x,t):=(-\Delta)^{-\frac{\alpha}{2}}u(x,t)$ for all $(x,t)\in \rr^{n+1}_+$. Then $\|v\|_{H^{\alpha,2}}=\|u\|_{\alpha,2,\star}<\infty.$
Hence, by Theorem \ref{t21}, there exists $h\in \mathcal{L}_{2,n+2\alpha}$ such that
$v(x,t)=e^{-t\sqrt{-\Delta}}h(x)$ and
\begin{equation}\label{31}
\|h\|_{\mathcal {L}_{2,n+2\alpha}}\lesssim \|v\|_{H^{\alpha,2}}\sim\|u\|_{\alpha,2,\star}<\infty.
\end{equation}
Notice that $u(x,t)=(-\Delta)^{\frac{\alpha}{2}}v(x,t)=(-\Delta)^{\frac{\alpha}{2}}e^{-t\sqrt{-\Delta}}h(x)$.

For simplicity, write for each $j\in \mathbb{N}$ that
$$
\lf\{\begin{array}{rl}
B & =B(x_0,r);\\
U_j(B)&=2^{j+1}B\setminus 2^{j}B=\{x\in\mathbb R^n: 2^j r\le |x-x_0|<2^{j+1}r\};\\
\chi_{U_j} &=\hbox{the\ characteristic\ function\ of}\ U_j.
\end{array}\r.
$$
Then one has
\begin{align*}
&r^{-(2\alpha+n)}
\int_{B(x_0,r)}\int_0^r |\nabla_{x,t} u(x,t)|^2\,{t^{1+2\alpha}\,dt\,dx}\\
&\quad\quad =r^{-(2\alpha+n)}
\int_{B}\int_0^r \lf|\nabla_{x,t}(t\sqrt{-\Delta})^\alpha e^{-t\sqrt{-\Delta}}(h-h_{2B})(x)\r|^2\,{t \,dt\,dx}\\
&\quad\quad =r^{-(2\alpha+n)}
\int_{B}\int_0^r \lf|\nabla_{x,t}(t\sqrt{-\Delta})^\alpha e^{-t\sqrt{-\Delta}}([h-h_{2B}]\chi_{2B})(x)\r|^2\,{t \,dt\,dx}\\
&\quad\qquad +\sum_{j=1}^\infty r^{-(2\alpha+n)}
\int_{B}\int_0^r \lf|\nabla_{x,t}(t\sqrt{-\Delta})^\alpha e^{-t\sqrt{-\Delta}}([h-h_{2B}]\chi_{U_j(B)})(x)\r|^2\,{t \,dt\,dx}\\
&\quad\quad := \mathrm{I}+\mathrm{II} \left(:=\sum_{j=1}^\infty\mathrm{II}_j\right).
\end{align*}

For the term $\mathrm{I}$, we use the standard $L^2$-boundedness to obtain
$$
\mathrm{I}\lesssim r^{-(2\alpha+n)}\int_\rn \big|\big((h-h_{2B})\chi_{2B}\big)(x)\big|^2\,{t \,dt\,dx}\lesssim \|h\|_{\mathcal{L}_{2,n+2\alpha}}^2.
$$

For the term $\mathrm{II}$, we notice the basic identity
$\nabla_{x,t} (t\sqrt{-\Delta})^\alpha e^{-t\sqrt{-\Delta}}(1)=0
$
and use the Fourier transform to derive that the integral kernel $\nabla_{x,t} K_\alpha(t,x,y)$ of
$(t\sqrt{-\Delta})^\alpha\nabla_{x,t} e^{-t\sqrt{-\Delta}}$ satisfies
$$
|\nabla_{x,t} K_\alpha(t,x,y)|\lesssim {t^{\alpha}}{({|x-y|}+{t})^{-(n+1+\alpha)}}, \quad \forall\, (x,y,t)\in\rr^n\times \rn\times (0,\infty).
$$
Consequently, an application of $h\in \mathcal{L}_{2,n+2\alpha}$ yields
\begin{align*}
\mathrm{II}
&\lesssim\sum_{j=1}^\infty r^{-(2\alpha+n)}\int_{B(x_0,r)}\int_0^r \Big|\int_{\rn}\nabla_{x,t}K_\alpha(t,x,y)\big((h-h_{2B})\chi_{U_j(B)}\big)(y)\,dy\Big|^2\,{t\,dt\,dx}\\
&\lesssim\sum_{j=1}^\infty r^{-(2\alpha+n)}\int_{B(x_0,r)}\int_0^r \Big|\int_{U_j(B)} {t^{\alpha}}{(t+{|x-y|})^{-(n+1+\alpha)}}\big((h-h_{2B})\chi_{U_j(B)}\big)(y)\,dy\Big|^2\,{t\,dt\,dx}\\
&\lesssim\sum_{j=1}^\infty r^{-(2\alpha+n)}\int_{B(x_0,r)}\int_0^r{t^{2\alpha}}{(2^jr)^{-2(n+1+\alpha)}}\Big|\int_{U_j(B)}\big((h-h_{2B})\chi_{U_j(B)}\big)(y)\,dy\Big|^2\,{t\,dt\,dx}\\
&\lesssim\sum_{j=1}^\infty {2^{-2j}}{(2^jr)^{-2(n+\alpha)}}\Big|\int_{U_j(B)}\big(h(y)-h_{2B}\big)\,dy\Big|^2\\
&\lesssim\sum_{j=1}^\infty {2^{-2j}}{(2^jr)^{-2(n+\alpha)}}\left[\int_{U_j(B)}\lf|h(y)-h_{2^{j+1}B}\r|\,dy+|U_j(B)|\sum_{i=1}^{j-1}|h_{2^iB}-h_{2^{i+1}B}|\right]^2\\
&\lesssim\sum_{j=1}^\infty {2^{-2j}}{(2^jr)^{-2(n+\alpha)}}\left[(2^jr)^{n+\alpha}+(2^jr)^n\sum_{i=1}^{j-1}(2^ir)^{\alpha}\right]^2\|h\|^2_{\mathcal{L}_{2,n+2\alpha}}.
\end{align*}

If $\alpha\in (-1,0)$, then
\begin{align*}
\mathrm{II}
&\lesssim\sum_{j=1}^\infty {2^{-2j}}{(2^jr)^{-2(n+\alpha)}}(2^jr)^{2n}r^{2\alpha}\|h\|^2_{\mathcal{L}_{2,n+2\alpha}}
\lesssim\sum_{j=1}^\infty 2^{-2j(1+\alpha)}\|h\|^2_{\mathcal{L}_{2,n+2\alpha}}\lesssim\|h\|^2_{\mathcal{L}_{2,n+2\alpha}}.
\end{align*}

If $\alpha=0$, then
\begin{align*}
\mathrm{II}
&\lesssim\sum_{j=1}^\infty {2^{-2j}}{(2^jr)^{-2n}}j^2(2^jr)^{2n}\|h\|^2_{\mathcal{L}_{2,n+2\alpha}}
\lesssim\sum_{j=1}^\infty 2^{-2j}j^2\|h\|^2_{\mathcal{L}_{2,n+2\alpha}}\lesssim\|h\|^2_{\mathcal{L}_{2,n+2\alpha}}.
\end{align*}

If $\alpha\in (0,1)$, then
\begin{align*}
\mathrm{II}
&\lesssim\sum_{j=1}^\infty {2^{-2j}}{(2^jr)^{-2(n+\alpha)}}(2^jr)^{2(n+\alpha)}\|h\|^2_{\mathcal{L}_{2,n+2\alpha}}
\lesssim\sum_{j=1}^\infty 2^{-2j}\|h\|^2_{\mathcal{L}_{2,n+2\alpha}}\lesssim\|h\|^2_{\mathcal{L}_{2,n+2\alpha}}.
\end{align*}

Putting the estimates of $\mathrm{I}$ and $\mathrm{II}$ together, 
we finally deduce the desired estimate from \eqref{31} that
$$
\|u\|_{\mathcal{H}^{\alpha,2}}\lesssim\|h\|_{\mathcal{L}_{2,n+2\alpha}}
\lesssim\|u\|_{\alpha,2,\star}.
$$

Conversely, let us show that
$\|u\|_{\alpha,2,\star}\lesssim \|u\|_{\mathcal{H}^{\alpha,2}}$. Two cases are considered as follows.

{\it Case 1}: $\alpha\in (-1,0]$. Under this condition, one has
$\|u\|_{H^{0,2}}\lesssim\|u\|_{\mathcal{H}^{\alpha,2}}<\infty$, i.\,e., $u\in H^{0,2}={\mathrm{HMO}}$, and so there is a function $f\in {\mathrm{BMO}}$ such that $u(x,t)=e^{-t\sqrt{-\Delta}}f(x)$ thanks to Theorem \ref{t21}.
By Lemma \ref{l31}(i) and \eqref{32}, this in turn implies that $f\in Q_{-\alpha}=(-\Delta)^{{\alpha}/{2}} \mathcal{L}_{2,n+2\alpha}$  and
$$\|u\|_{\alpha,2,\star}\lesssim \|f\|_{(-\Delta)^{{\alpha}/{2}} \mathcal{L}_{2,n+2\alpha}}\lesssim \|u\|_{\mathcal{H}^{\alpha,2}}.$$

{\it Case 2}: $\alpha\in (0,1)$. Under this condition, one uses the mean value property for sub-harmonic functions to deduce that, for
all $(x,t)\in \rr^{n+1}_+$,
\begin{equation*}
|\nabla_{x,t} u(x,t)|\lesssim \left(t^{-(2\alpha+n+2)}\int_{B(x_0,t)}\int_{2^{-1}t}^{2t}|\nabla_{x,s} u(x,s)|^2\,s^{1+2\alpha}\,ds\,dx\right)^\frac12 \lesssim t^{-1}\|u\|_{\mathcal{H}^{\alpha,2}}.
\end{equation*}
Notice that, since $u$ is harmonic, $u(x,t+s)=e^{-s\sqrt{-\Delta}}u(x,t)$ for $t,s>0$.
These, along with
$$(-\Delta)^{-\frac{\alpha}{2}}u(x,t)=\big[2\pi\Gamma(\alpha)\big]^{-1}\int_0^\infty u(x,t+s)\,s^{\alpha-1}ds,
$$
$t,s>0$ and $\alpha\in (0,1)$, derive
\begin{align*}
&\int_0^r\int_{B(x_0,r)}|\nabla_{x,t} (-\Delta)^{-\frac{\alpha}{2}}u(x,t)|^2 t\,dx\,dt\\
&\hspace{1cm}\lesssim\int_0^r\int_{B(x_0,r)}\left(\int_0^\infty \left|\nabla_{x,t}  u(x,t+s)\right|s^{\alpha-1}\,ds\right)^2 t\,dx\,dt\\
&\hspace{1cm}\lesssim
\|u\|^2_{\mathcal{H}^{\alpha,2}}\int_0^r\int_{B(x_0,r)}\left(\Big(\int_0^t +\int_t^\infty\Big) (t+s)^{-1}s^{\alpha-1}\,ds\right)^2 t\,dx\,dt\\
&\hspace{1cm}\lesssim\|u\|_{\mathcal{H}^{\alpha,2}}^2r^n\int_0^r\left(\Big(\int_0^t s^{\alpha-1}\,ds\Big)^2t^{-1}+\Big(\int_t^\infty s^{\alpha-2}\,ds\Big)^2 t\right)\,dt\\
&\hspace{1cm}\lesssim r^{n+2\alpha}\|u\|_{\mathcal{H}^{\alpha,2}}^2,
\end{align*}
thereby yielding $\|u\|_{\alpha,2,\star}\lesssim \|u\|_{\mathcal{H}^{\alpha,2}}$.
\end{proof}

We next prove Theorem \ref{t31}, which is the scaling invariant counterpart of Theorem \ref{t21}.

\begin{proof}[Proof of Theorem \ref{t31}] (i) Let us first show that
$$\mathcal{H}^{\alpha,2}\supseteq e^{-t\sqrt{-\Delta}}(-\Delta)^\frac{\alpha}{2}\mathcal{L}_{2,n+2\alpha}.$$
Given an $f\in (-\Delta)^\frac{\alpha}{2}\mathcal{L}_{2,n+2\alpha}$,
by the definition, we have  $(-\Delta)^{-\frac{\alpha}{2}}f\in \mathcal{L}_{2,n+2\alpha}$. From Lemma \ref{l33} and Theorem \ref{t21},
  it follows that
  $$\lf\|e^{-\sqrt{-\Delta}}f\r\|_{\mathcal{H}^{\alpha,2}} \sim\lf \|e^{-\sqrt{-\Delta}}f\r\|_{\alpha,2,\star}\sim \lf\|  e^{-\sqrt{-\Delta}}(-\Delta)^{-\frac{\alpha}{2}}f\r\|_{H^{\alpha,2}}\lesssim \|f\|_{(-\Delta)^\frac{\alpha}{2}\mathcal{L}_{2,n+2\alpha}},$$
  and hence $\mathcal{H}^{\alpha,2}\supseteq e^{-t\sqrt{-\Delta}}(-\Delta)^\frac{\alpha}{2}\mathcal{L}_{2,n+2\alpha}$.

Reversely, suppose
$$
u\in \mathcal{H}^{\alpha,2}\ \ \text{and}\ \ v(x,t):=(-\Delta)^{-\frac{\alpha}{2}}u(x,t),\quad \forall\,(x,t)\in \rr^{n+1}_+.
$$
The formula on $(-\Delta)^{-\frac{\alpha}{2}}u(x,t)$ presented in the proof of Lemma \ref{l32}(i) indicates that $v$ is harmonic in $\mathbb R^{n+1}_+$, and so $v$ belongs to $H^{\alpha,2}$. According to Theorem \ref{t21}, there is a function $g\in\mathcal{L}_{2,n+2\alpha}$ such that
$$v(x,t)=e^{-t\sqrt{-\Delta}} g(x).$$
Letting $f=(-\Delta)^\frac{\alpha}{2}g$, one obtains $u(x,t)=e^{-t\sqrt{-\Delta}} f(x)$ and $f\in(-\Delta)^\frac{\alpha}{2}\mathcal{L}_{2,n+2\alpha}$. So, one finds $u\in e^{-t\sqrt{-\Delta}}(-\Delta)^\frac{\alpha}{2}\mathcal{L}_{2,n+2\alpha}$, as desired.

(ii) Let $\alpha\in (0,1)$. Suppose $u\in \mathop\mathrm{HB}$, then, for each $(x_0,r)\in \rr^{n+1}_+$, it holds true that
\begin{eqnarray*}
&&r^{-(2\alpha+n)}\int_{B(x_0,r)}
\int_0^r|\nabla_{x,t} u(x,t)|^2\,t^{1+2\alpha} \,dt\,dx\\
&&\quad \le \lf[\sup_{(x,t)\in\mathbb R^{n+1}_+}t|\nabla_{x,t}u(x,t)| \r]^2
r^{-(2\alpha+n)}\int_{B(x_0,r)}
\int_0^r t^{2\alpha-1} \,dt\,dx\\
&&\quad\lesssim \|u\|_{\mathop\mathrm{HB}}^2,
\end{eqnarray*}
and hence $\mathop\mathrm{HB}\subseteq \mathcal{H}^{\alpha,2}$.

Conversely, let $u\in \mathcal{H}^{\alpha,2}$. Then, by the mean value property of subharmonic function, we conclude that,
for each $(x,t)\in \rr^{n+1}_+$,
\begin{eqnarray*}
t|\nabla_{x,t}u(x,t)| &&\lesssim  t^{-n}\int_{B(x_0,t/2)}
\int_{t/2}^{3t/2}|\nabla_{x,t} u(x,s)| \,ds\,dx\\
&&\lesssim\lf(\sup_{(x_0,r)\in\mathbb R^{n+1}_+}r^{-(2\alpha+n)}\int_{B(x_0,r)}
\int_0^r|\nabla_{x,t} u(x,s)|^2\,s^{1+2\alpha}\,ds\,dx\r)^{1/2}\\
&&\lesssim \|u\|_{\mathcal{H}^{\alpha,2}},
\end{eqnarray*}
which implies that $\mathcal{H}^{\alpha,2}\subseteq \mathop\mathrm{HB}$, and hence completes the proof of Theorem \ref{t31}.
\end{proof}

\begin{rem}\label{r31} \rm Now, the space $(-\Delta)^{\frac{\alpha}{2}}\mathcal{L}_{2,n+2\alpha}=(-\Delta)^{\frac{\alpha}{2}}{\mathop\mathrm{Lip}}\,\alpha$ with $\alpha\in (0,1)$ exists as the trace of the harmonic Bloch space $\mathop\mathrm{HB}$. From Lemma \ref{l32}(iii) and Theorem \ref{t31}(ii), we deduce the following inclusion chain:
$$
\mathcal{H}^{-1,2}\subseteq \mathcal{H}^{-\alpha,2}\subseteq {\mathrm{HMO}}\subseteq \mathcal{H}^{\alpha,2}=\mathop\mathrm{HB},\quad\forall\, \alpha\in (0,1).
$$
Also, according to \cite[Theorem 3.3]{Str1}, one has
\begin{align*}
f\in(-\Delta)^{-\frac{\alpha}{2}}{\mathrm{BMO}}&\Longleftrightarrow\sup_{(x_0,r)\in\mathbb R^{n+1}_+}r^{-n}\iint_{B(x_0,r)\times B(x_0,r)}|f(x)-f(y)|^2|x-y|^{-(n+2\alpha)}\,dx\,dy<\infty\\
&\Longrightarrow f\in \mathop\mathrm{Lip}\alpha.
\end{align*}

\end{rem}

\section{Connections to Caloric Functions and Navier-Stokes Equations}\label{s4}
\hskip\parindent
In this section, we outline the steps necessary to apply the techniques of Sections \ref{s2} and \ref{s3}
to establish the analogues for caloric functions, i.\,e., solutions to the heat equation, and naturally
handling a well/ill-posedness of the incompressible Navier-Stokes system on $\mathbb R^{3+1}_+$.

\begin{proof}[Proof of Theorem \ref{t41}] (i) Similar to Lemma \ref{l21},
in light of the argument for \cite[Lemma 2.1]{Xiao1} when $\alpha\in (-1,0]$ and \cite[p.19]{FJW} when $\alpha\in (0,1)$,
we find that
$$
f\in \mathcal{L}_{2,n+2\alpha}\Longleftrightarrow \sup_{(x_0,r)\in \mathbb R^{n+1}_+}r^{-(2\alpha+n)}\int_{B(x_0,r)}\int_0^{r^2}|\nabla e^{t\Delta}f(x)|^2\,dt\,dx<\infty,
$$
thereby obtaining $T^{\alpha,2}\supseteq e^{t\Delta}\mathcal{L}_{2,n+2\alpha}$.

To establish the reversed inclusion of this last one, we simply follow those three steps in the proof of Theorem \ref{t21}, but this time, with $k^{-1}$, $e^{-t\sqrt{-\Delta}}(x,y)$ and $\nabla_{x,t}$ being replaced by $k^{-2}$, $e^{t\Delta}(x,y)$ and $\nabla_{x,t}:=\nabla$ respectively, plus keeping $\int_{\mathbb R^n}\partial_j e^{t\Delta}(x,0)dx=0$ in mind, and so leave the details  to the interested reader.

(ii) Similar to Lemma \ref{l31}, one has the readily-verified equivalence below (see \cite{Xiao1, Xiao2} for $\alpha\in (-1,0]$; \cite[p. 19]{FJW} and Remark \ref{r31} for $\alpha\in (0,1)$):
\begin{align*}
\|f\|^2_{(-\Delta)^\frac{\alpha}{2}\mathcal{L}_{2,n+2\alpha}}&\sim\sup_{(x_0,r)\in \mathbb R^{n+1}_+}r^{-(2\alpha+n)}\int_{B(x_0,r)}\int_0^{r^2}|\nabla e^{t\Delta}f(x)|^2t^\alpha\,dt\,dx\\
&\sim \sup_{(x_0,r)\in \mathbb R^{n+1}_+}r^{-(2\alpha+n)}\int_{B(x_0,r)}\int_0^{r^2}|\nabla (-\Delta)^{-\frac{\alpha}{2}}e^{t\Delta}f(x)|^2\,dt\,dx.
\end{align*}

Moreover, by an argument similar to that used in the proof of Lemma \ref{l33}, we conclude that, for all caloric functions $u$
on $\rr^{n+1}_+$, it holds true that
$$
\|u\|_{\mathcal{T}^{\alpha,2}}\sim\|u\|_{\alpha,2,\dagger}:=\left(\sup_{(x_0,r)\in\mathbb R^{n+1}_+}r^{-(2\alpha+n)}\int_{B(x_0,r)}\int_0^r|\nabla_{x,t}(-\Delta)^{-\alpha/2} u(x,t)|^2\,t\,dt\,dx\right)^\frac12.
$$

Thus, combining the above two facts, one sees that, for $f\in (-\Delta)^\frac{\alpha}{2}\mathcal{L}_{2,n+2\alpha}$, it holds true that
\begin{align*}
\|e^{-t\Delta}f\|_{\mathcal{T}^{\alpha,2}}\sim\|e^{-t\Delta}f\|_{\alpha,2,\dagger}
\sim\|f\|_{(-\Delta)^\frac{\alpha}{2}\mathcal{L}_{2,n+2\alpha}},
\end{align*}
which implies that $e^{t\Delta}(-\Delta)^\frac{\alpha}{2}\mathcal{L}_{2,n+2\alpha}\subseteq \mathcal{T}^{\alpha,2}$.

On the other hand, suppose
$$
u\in \mathcal{T}^{\alpha,2}\ \ \text{and}\ \ w(x,t):=(-\Delta)^{-\frac{\alpha}{2}}u(x,t), \quad \forall (x,t)\in \rr^{n+1}_+.
$$
The formula on $(-\Delta)^{-\frac{\alpha}{2}}u(x,t)$ presented in the proof of Lemma \ref{l32}(i) illustrates that $w$ is caloric in $\mathbb R^{n+1}_+$, and so $w$ belongs to $T^{\alpha,2}$. According to (i) as above, there exists a function $g\in\mathcal{L}_{2,n+2\alpha}$ such that $w(x,t)=e^{t\Delta}g(x)$. Putting $f=(-\Delta)^\frac{\alpha}{2}g$, one has $u(x,t)=e^{t\Delta} f(x)$ and $f\in(-\Delta)^\frac{\alpha}{2}\mathcal{L}_{2,n+2\alpha}$. So, one finds $u\in e^{t\Delta}(-\Delta)^\frac{\alpha}{2}\mathcal{L}_{2,n+2\alpha}$, whence reaching $e^{t\Delta}(-\Delta)^\frac{\alpha}{2}\mathcal{L}_{2,n+2\alpha}\supseteq \mathcal{T}^{\alpha,2}$. Therefore, the identity in (ii) holds true.

(iii) Given $\beta\in (0,1)$, in a way similar to that used in the proof of \cite[(1.1)]{FN}, one employs Theorem \ref{t41}(ii) to deduce
$$
u\in \mathcal{T}^{\beta,2}\Longrightarrow
|\nabla u(x,t)|\lesssim t^{-\frac12}\|u\|_{\beta,2,\dagger}\lesssim t^{-\frac12}\|u\|_{\mathcal{T}^{\beta,2}},\ \ \forall\,(x,t)\in\mathbb R^{n+1}_+.
$$
Thus, $\mathcal{T}^{\beta,2}\subseteq \mathop\mathrm{CB}$. Conversely, if $u\in \mathop\mathrm{CB}$, then
$$
\int_0^{r^2}\int_{B(x_0,r)}|\nabla u(x,t)|^2t^\beta\,dx\,dt\lesssim \|u\|^2_{\mathop\mathrm{CB}} \int_0^{r^2}\int_{B(x_0,r)}t^{\beta-1} \,dx\,dt
\lesssim r^{n+2\beta}\|u\|^2_{\mathop\mathrm{CB}},
$$
and hence $u\in \mathcal{T}^{\beta,2}$ follows from Theorem \ref{t41}(ii). Therefore, $\mathcal{T}^{\beta,2}=\mathop\mathrm{CB}$.
\end{proof}

\begin{rem}\rm
\label{r41} Naturally, one has the following inclusion chain:
$$
\mathcal{T}^{-1,2}\subseteq \mathcal{T}^{-\alpha,2}\subseteq {\mathrm{TMO}}\subseteq \mathcal{T}^{\alpha,2}=\mathop\mathrm{CB},\quad\forall\,\alpha\in (0,1).
$$
\end{rem}

Upon recalling the invariance of the Navier-Stokes system (3D-N-S) under the scaling transform
$$
\left\{\begin{array}{rl}
{\bf u}(x,t)\mapsto {\bf u}_\lambda(x,t):=&\lambda {\bf u}(\lambda x,\lambda^2t);\\
{\bf a}(x)\mapsto {\bf a}_\lambda(x):=&\lambda {\bf a}(\lambda x);\\
p(x,t)\mapsto p_\lambda(x,t):=&\lambda^2p(\lambda x,\lambda^2t),\\
\end{array}
\right.
$$
that is to say, if $({\bf u}, {\bf a}, p)$ obeys the Navier-Stokes system (3D-N-S), then
$({\bf u}_\lambda, {\bf a}_\lambda, p_\lambda)$ also obeys the Navier-Stokes system (3D-N-S) for any $\lambda>0$, we can adapt the methods described in Theorem \ref{t41}(ii) and its proof to treat the well/ill-posedness of the Cauchy problem for the incompressible Navier-Stokes system in $\mathbb R^{3+1}_+$; see e.g. \cite{KT, Xiao1, Xiao2, LL, CW, JS, M-S, GP, BP,Tri, Y} for more information.

Let $\dot{B}^{-1,\infty}_\infty\equiv \dot{B}^{-1,\infty}_\infty(\rn)$ be 
the Besov space of all measurable functions $f$ on $\mathbb R^n$ satisfying
$$
\|f\|_{\dot{B}^{-1,\infty}_{\infty}}:=\sup_{(x,t)\in\mathbb R^{n+1}_+}\sqrt{t}|e^{t\Delta}f(x)|<\infty.
$$
\begin{thm}\label{t42}

{\rm(i)} If $\alpha\in (-1,0]$, then, for each $T\in (0,\infty]$,
$\big((-\Delta)^\frac{\alpha}{2}\mathcal{L}_{2,2\alpha+n}\big)_T^{-1}
=Q_{-\alpha,T}^{-1}.
$

{\rm(ii)} If $\alpha\in (0,1)$, then 
$
\big((-\Delta)^\frac{\alpha}{2}\mathcal{L}_{2,2\alpha+n}\big)_\infty^{-1}=\dot{B}^{-1,\infty}_{\infty}.
$
\end{thm}
\begin{proof} (i) If $\alpha=0$, then by the definition, one has
$(-\Delta)^\frac{\alpha}{2}\mathcal{L}_{2,2\alpha+n}=\mathrm{BMO}=Q_{0}$,
and the conclusion is identical.

If $\alpha\in (-1,0)$, then it follows, from \cite{WX,Xiao1}, that
$(-\Delta)^\frac{\alpha}{2}\mathcal{L}_{2,2\alpha+n}=Q_{-\alpha}$ and hence, for each $T\in (0,T]$, one has
$$\big((-\Delta)^\frac{\alpha}{2}\mathcal{L}_{2,2\alpha+n}\big)_T^{-1}
=Q_{-\alpha,T}^{-1}.$$

(ii) Let $\alpha\in (0,1)$. Notice that, if $f\in \dot{B}^{-1,\infty}_\infty$, then we have
$$
\|f\|_{n+2\alpha,-1,\infty}\lesssim \|f\|_{\dot{B}^{-1,\infty}_{\infty}}\left[\sup_{(x_0,r)\in\mathbb R^{n+1}_+}r^{-(2\alpha+n)}{\int_0^{r^2}\lf(\int_{B(x_0,r)}t^{-1}\,dy\r)\,t^{\alpha}dt}\right]^\frac12\lesssim\|f\|_{\dot{B}^{-1,\infty}_{\infty}}.
$$
Thus, $\dot{B}^{-1,\infty}_{\infty}$ is contained in $ \big((-\Delta)^\frac{\alpha}{2}\mathcal{L}_{2,n+2\alpha}\big)_\infty^{-1}$. 

Conversely, since $\|\cdot\|_{n+2\alpha,-1,\infty}$ enjoys the following scale invariance
$$
\|\lambda f(\lambda\cdot-x_0)\|_{n+2\alpha,-1,\infty}=\|f\|_{n+2\alpha,-1,\infty},\quad\forall\,(x_0,\lambda)\in\mathbb R^{n+1}_+,
$$
the space $\big((-\Delta)^\frac{\alpha}{2}\mathcal{L}_{2,n+2\alpha}\big)_\infty^{-1}$ must be contained in $\dot{B}^{-1,\infty}_{\infty}$ which is the maximal space in the sense of the above scale invariance; see e.g. \cite[Proposition 7]{Can}.
The proof of Theorem \ref{t42} is then completed.
\end{proof}

The following corollary follows from combining Theorem \ref{t42} and 
the known conclusions from \cite{BP,KT,Xiao1,Xiao2}.

\begin{cor}\label{c41} For $\alpha\in (-1,1)$, one has

{\rm(i)} The well-posedness under $\alpha\in (-1,0]$, precisely, for any $T\in (0,\infty]$, there is $\delta>0$ such that the Navier-Stokes system (3D-N-S) has a unique mild solution ${\bf u}$ with $\|{\bf u}\|_{X_{\alpha,T}}<\infty$ for all initial data ${\bf a}$ with $\|{\bf a}\|_{3+2\alpha,-1,T}\le\delta$.

{\rm(ii)} The ill-posedness under $\alpha\in (0,1)$, precisely, for any $\epsilon>0$, there exists an initial data ${\bf a}$ with $\|{\bf a}\|_{3+2\alpha,-1,\infty}\le\epsilon$ such that the corresponding mild solution ${\bf u}$ to the Navier-Stokes system (3D-N-S) obeys
$\|{\bf u}\|_{X_{\alpha,T_0}}\ge\epsilon^{-1}$ for some $T_0\in (0,\epsilon)$.
\end{cor}

\begin{proof} (i) Thanks to Theorem \ref{t42}(i),
$$
\big((-\Delta)^\frac{\alpha}{2}\mathcal{L}_{2,2\alpha+n}\big)_T^{-1}
=Q_{-\alpha,T}^{-1},\quad\forall\,\alpha\in (-1,0],
$$
the desired well-posedness follows from \cite{KT} and \cite[Theorem 2.5]{Xiao2} (cf. \cite{Xiao1}).

(ii) According to \cite[Theorem 1.1]{BP} and due to the basic estimate below
\begin{align*}
\sup_{(x,t)\in\mathbb R^{3}\times (0,T_0)}\sqrt{t}|{\bf u}(x,t)|
&\lesssim\|{\bf u}\|_{X_{\alpha,T_0}}\\
&\lesssim \sup_{(x,t)\in\mathbb R^3\times(0,T_0)}\sqrt{t}|{\bf u}(x,t)|\left(1+\sup_{(x_0,r)\in\mathbb R^{3}\times(0,T_0)}\lf[r^{-2\alpha}{\int_0^{r^2}\,t^{\alpha-1}dt}\r)^\frac12\right]\\
&\lesssim \sup_{(x,t)\in\mathbb R^{3}\times(0,T_0)}\sqrt{t}|{\bf u}(x,t)|,\quad\forall\,\alpha\in (0,1),
\end{align*}
to show Theorem \ref{t42}(ii), it suffices to prove
$$
\big((-\Delta)^\frac{\alpha}{2}\mathcal{L}_{2,2\alpha+n}\big)_\infty^{-1}=\dot{B}^{-1,\infty}_{\infty},\quad\forall\,\alpha\in (0,1),
$$
which is just Theorem \ref{t42}(ii). The proof is then completed.
\end{proof}

\begin{rem}\rm
\label{r42} The arguments for Theorem \ref{t42} and \cite[Theorem 1.2 (iii)]{Xiao1} or \cite[Theorem 2.1]{Xiao2} actually reveal
$$
\big((-\Delta)^\frac{\alpha}{2}\mathcal{L}_{2,n+2\alpha}\big)_\infty^{-1}
=\nabla\cdot\big((-\Delta)^\frac{\alpha}{2}\mathcal{L}_{2,n+2\alpha}\big)^n,
\quad\forall\,\alpha\in (-1,1).
$$
This, along with the proof of Theorem \ref{t42}, especially implies that
$$
\big((-\Delta)^{-\frac{\alpha}{2}}\mathcal{L}_{2,n-2\alpha}\big)^{-1}_\infty=Q_{\alpha,\infty}^{-1}\subseteq {\mathrm{BMO}}^{-1}\subseteq \dot{B}^{-1,\infty}_\infty=\big((-\Delta)^\frac{\alpha}{2}\mathcal{L}_{2,n+2\alpha}\big)_\infty^{-1},\quad\forall\,\alpha\in (0,1).
$$
Such a series of function space embeddings, plus the invariance of the Navier-Stokes system (3D-N-S), nicely explains that the phenomenon happened in Theorem \ref{t42} is really natural and hence important.
\end{rem}

\noindent Renjin Jiang$^{1}$, Jie Xiao$^{2}$ and Dachun Yang$^{1}$

\

\noindent
1. School of Mathematical Sciences, Beijing Normal University,
Laboratory of Mathematics and Complex Systems, Beijing 100875, People's Republic of China

\

\noindent
2. Department of Mathematics and Statistics, Memorial University, St. John's, NL A1C 5S7, Canada

\

\noindent{\it E-mail addresses}:
\texttt{rejiang@bnu.edu.cn}

\hspace{2.3cm}
\texttt{jxiao@mun.ca}

\hspace{2.3cm}
\texttt{dcyang@bnu.edu.cn}

\end{document}